\documentclass[reqno]{amsart}
\usepackage{amsthm}
\usepackage{amsmath}
\usepackage{amssymb}
\usepackage{graphicx}
\usepackage{calrsfs}
\usepackage{enumerate}
\usepackage{url}

\newtheorem{theorem}{Theorem}[section]

\numberwithin{equation}{section}

\newcommand{\abs}[1]{\left\lvert#1\right\rvert}

\newcommand{\gfextn}[0]{jpeg}
\newcommand{\oscwidth}[0]{1.55in}
\DeclareMathOperator{\Res}{Res}

\begin{document}

\title{Oscillations in weighted arithmetic sums}

\author[M.~J. Mossinghoff]{Michael J. Mossinghoff}
\address{Center for Communications Research\\
Princeton, NJ, USA}
\email{m.mossinghoff@idaccr.org}

\author[T.~S. Trudgian]{Timothy S. Trudgian}\thanks{This work was supported by a Future Fellowship (FT160100094 to T.~S. Trudgian) from the Australian Research Council.}
\address{School of Science\\
UNSW Canberra at ADFA\\
ACT 2610, Australia}
\email{t.trudgian@adfa.edu.au}

\date{\today}

\begin{abstract}
We examine oscillations in a number of sums of arithmetic functions involving $\Omega(n)$, the total number of prime factors of $n$, and $\omega(n)$, the number of distinct prime factors of $n$.
In particular, we examine oscillations in $S_\alpha(x) = \sum_{n\leq x} (-1)^{n - \Omega(n)}/n^\alpha$ and in  $H_\alpha(x) = \sum_{n\leq x} (-1)^{\omega(n)}/n^{\alpha}$ for $\alpha\in[0,1]$, and in $W(x)=\sum_{n\leq x} (-2)^{\Omega(n)}$.
We show for example that each of the inequalities $S_0(x)<0$, $S_0(x)>3.3\sqrt{x}$, $S_1(x)>0$, and $S_1(x)\sqrt{x}<-3.3$ is true infinitely often, disproving some hypotheses of Sun.
\end{abstract}

\maketitle

\section{Introduction}\label{Intro}

Let $\omega(n)$ denote the number of distinct prime factors of $n$, and let $\Omega(n)$ denote its total number of prime factors, counting multiplicity.
Questions involving the parity of these functions have a long history.
For $\alpha\in[0,1]$, define
\begin{equation}\label{eqnL}
L_{\alpha}(x) = \sum_{n\leq x} \frac{(-1)^{\Omega(n)}}{n^{\alpha}}
\end{equation}
and
\begin{equation}\label{eqnH}
H_{\alpha}(x) = \sum_{n\leq x} \frac{(-1)^{\omega(n)}}{n^{\alpha}}.
\end{equation}
P\'olya studied $L_0(x)$ in his 1919 article \cite{Polya}.
He noted that $L_0(x)$ was not positive after $L_0(1)=1$ up to approximately $x=1500$, and remarked that the Riemann hypothesis, as well as the simplicity of the zeros of the zeta function, would follow if this held true for all sufficiently large $x$.
More generally, it is well known that both of these statements would hold if the normalized function $L_0(x)/\sqrt{x}$ were bounded by some constant, either above or below.
In his influential 1942 paper \cite{Ingham1942}, Ingham showed that considerably more would follow in this case: there would exist infinitely many linear dependencies over the rationals among the ordinates of the zeros of the zeta function on the critical line in the upper half plane.
Stated another way, if the (positive) ordinates of the nontrivial zeros of $\zeta(s)$ are linearly independent, then arbitrarily large oscillations must exist in $L_0(x)/\sqrt{x}$.
In fact, Ingham showed that this holds both for P\'olya's function $L_0(x)$ and for its close cousin, the Mertens function,
\begin{equation}\label{eqnMertens}
M(x)=\sum_{n\leq x} \mu(n).
\end{equation}
Since it seems there is no \textit{a priori} reason to suspect such linear dependencies, and none has ever been detected among the zeros of the zeta function, it is often conjectured that these functions exhibit unbounded oscillations.
Large oscillations have been shown to exist in the Mertens function \cite{Best,Hurst} and in the error term for the distribution of $k$-free numbers \cite{Olive}.

Tur\'an studied $L_1(x)$ in 1948 \cite{Turan}, where he reported that this function is never negative over $2\leq x\leq1000$, and connected  its behavior to the zeros of the Riemann zeta function.
Here, the Riemann hypothesis, as well as the simplicity of the zeros and the existence of linear dependencies among their ordinates, would similarly follow if $L_1(x)\sqrt{x}$ were bounded either above or below.
Haselgrove \cite{Haselgrove} established that infinitely many sign changes in both $L_0(x)$ and $L_1(x)$ exist. Specific locations of sign changes in these functions were first determined respectively in \cite{Lehman} and in~\cite{Borwein}.

The broader family $L_\alpha(x)$ was investigated in \cite{MosTru}, where the authors established how these functions are connected to the Riemann hypothesis, as well as the statement on the simplicity of the zeros of $\zeta(s)$ and the question of linear independence.
In \cite{MT2017} it was shown that substantial oscillations must exist in the normalized functions $L_{\alpha}(x)x^{\frac{1}{2}-\alpha}$.
For example, it was shown that each of the inequalities $L_0(x)/\sqrt{x}>1$ and $L_1(x)\sqrt{x}>2.37$ holds for infinitely many integers $x$ (see \cite[Thms.\ 1.1 and 6.2]{MT2017}).
The reader is referred to \cite{Borwein,HumphriesJNT,MosTru,MT2017} for additional information and background on this family of functions.

In \cite{Dickens}, the authors established similar connections and statements for the function $H_0(x)$.
Here, among other things, it was established that each of the inequalities $H_0(x)/\sqrt{x} > 1.7 $ and $H_0(x)/\sqrt{x} < -1.7$ holds for infinitely many positive integers $x$.

In \cite[Conj.\ 1.1]{Sun}, Sun hypothesized that both $L_1(x)$ and $H_1(x)$ are $O_{\epsilon}(x^{-\frac{1}{2} + \epsilon})$ for any $\epsilon>0$.
Both of these statements seem likely to be true, since either of these is equivalent to the Riemann hypothesis (see also \cite{OBMR,MosTru}).
Sun also remarked there that $H_1(x)$ may be $O(x^{-\frac{1}{2}})$, but by the same argument employed in \cite[\S5]{MT2017} for $H_0(x)/\sqrt{x}$, bounding $H_1(x)\sqrt{x}$ would again imply the existence of infinitely many linear dependence relations among the ordinates of the zeros of the zeta function in the upper half plane, so a bound of this strength seems doubtful.

In this paper, we prove that large oscillations must exist in $H_{\alpha}(x)x^{\alpha-\frac{1}{2}}$ for every $\alpha\in[0, 1]$.
This presents an $\omega$-complement to the $\Omega$ work of \cite{MosTru} and \cite{MT2017}, and generalizes the oscillation result concerning the case $\alpha=0$ from \cite{Dickens} in a natural way.
For $\sigma>1$ and $s=\sigma+it$, let $h(s)$ denote the Dirichlet series
\begin{equation}\label{eqnhDef}
h(s) = \sum_{n\geq 1} \frac{(-1)^{\omega(n)}}{n^s} = \prod_p \biggl(1-\sum_{k\geq1}\frac{1}{p^{ks}}\biggr) = \zeta(s)\prod_p\left(1-\frac{2}{p^s}\right).
\end{equation}
The series $h(1)$ was investigated by van de Lune and Dressler in 1975 \cite{Dressler}, who proved that its value is $0$.
Following \cite[\S5]{Dickens}, we may write the Euler product above as
\begin{equation}\label{chalk}
h(s) = \frac{F_6(s)}{Z_6(s)},
\end{equation}
where
\begin{equation*}\label{eqnF6}
F_{6}(s) = \prod_{p} \left(1 - \frac{18}{p^{7s}} - \frac{30}{p^{8s}} - \frac{56}{p^{9s}} - \cdots \right)
\end{equation*}
is an explicitly given function that converges in the half plane $\sigma>1/7$,
and
\begin{equation*}\label{eqnZ6}
Z_6(s) = \zeta(s) \zeta(2s)\zeta^{2}(3s) \zeta^{3}(4s) \zeta^{6}(5s) \zeta^{9}(6s),
\end{equation*}
so that $h(s)$ can be continued analytically to the left of the $\sigma=1$ line.
(We remark as in \cite{Dickens} that there is no trouble in extending the form in \eqref{chalk} so that the corresponding function $F_{k}(s)$ converges on $\sigma > 1/(k+1)$ for larger values of $k$.)
For convenience we write $h(\alpha)$ in place of $F_6(\alpha)/Z_6(\alpha)$ for $\alpha\in(\frac{1}{2},1)$ throughout.
For $0\leq \alpha\leq 1$, we define $\mathcal{H}_{a}(x)$ by
\begin{equation}\label{eqnHcal}
\mathcal{H}_{a}(x)=
\begin{cases}  H_{\alpha}(x), & \textrm{if $0 \leq \alpha \leq \frac{1}{2}$},\\
H_{\alpha}(x) - h(\alpha), & \textrm{if $\frac{1}{2} <\alpha \leq 1$}.
\end{cases}
\end{equation}
We establish the following theorem.

\begin{theorem}\label{thmHosc}
For each $\alpha\in[0,1]$, each of the following inequalities holds for infinitely many positive integers $x$:
\[
\mathcal{H}_\alpha(x) x^{\alpha-\frac{1}{2}} < -1.7, \quad \mathcal{H}_{\alpha}(x) x^{\alpha-\frac{1}{2}} > 1.7.
\]
\end{theorem}

Sun made some additional conjectures regarding sums involving $\Omega(n)$ in \cite{Sun}, which we also address here.
Let
\begin{equation}\label{eqnSalpha}
S_\alpha(x) = \sum_{n\leq x}\frac{(-1)^{n - \Omega(n)}}{n^\alpha}.
\end{equation}
In \cite[Hypoth.\ 1.1]{Sun}, Sun made the following four hypotheses:
\begin{gather}
S_0(x) > 0 \textrm{\;\ for $x\geq 5$,}\label{Sol1}\\
S_1(x) < 0 \textrm{\;\ for $x\geq 1$,}\label{Sol1b}\\
1 < \frac{S_0(x)}{\sqrt{x}} < 2.3 \textrm{\;\ for $x\geq 325$,}\label{Sol2}\\
-2.3 < S_1(x)\sqrt{x} < -1 \textrm{\;\ for $x\geq 3$.}\label{Sol3}
\end{gather}
Certainly \eqref{Sol2} and \eqref{Sol3} are stronger than \eqref{Sol1} and \eqref{Sol1b} respectively for sufficiently large $x$, and clearly \eqref{Sol2} and \eqref{Sol3} appear related.
Sun reported that \eqref{Sol2} was checked for $x\leq10^{11}$ and \eqref{Sol3} for $x\leq2\cdot 10^{9}$.

We disprove all four of these statements in this article.
In two of the conjectured inequalities above---the lower bounds in \eqref{Sol2} and \eqref{Sol3}---we also determine the minimal counterexample, after calculating these functions for $x\leq7.5\cdot10^{14}$.
We summarize some of the results of our computations in the following theorem.

\begin{theorem}\label{thmSComps}
The following statements hold for the function $S_{\alpha}(x)$ from \eqref{eqnSalpha}.
\begin{enumerate}[(a)]
\item The smallest integer $x\geq325$ for which $S_0(x) \leq \sqrt{x}$ is $x_0=702469704523413$.
Here, $S_0(x_0)=26504145$, so $\sqrt{x_0}-S_0(x_0) = 0.044189\ldots$\,.
\item The minimal value of $S_0(x)/\sqrt{x}$ for integers $x$ with $324\leq x\leq 7.5\cdot 10^{14}$ is $0.9876\ldots$\,, occurring at $x_0^*=702494078543084$, and $S_0(x_0^*)=26175950$.
\item The smallest integer $x\geq1$ for which $S_1(x)\sqrt{x} \leq -2.3$ is $x_1=702475225213517$, at which $S_1(x_1)\sqrt{x_1}=-2.300000012\ldots$\,.
\item The minimal value of $S_1(x)\sqrt{x}$ for integers $x\leq 7.5\cdot 10^{14}$ occurs at $x_1^*=702494078542878$, and $S_1(x_1^*)\sqrt{x_1^*} = -2.3108948\ldots$\,.
\end{enumerate}
\end{theorem}

In addition, we determine lower bounds on the oscillations exhibited by $S_0(x)/\sqrt{x}$ and $S_1(x)\sqrt{x}$.
These results show that each of the conjectured bounds in \eqref{Sol1}, \eqref{Sol1b}, \eqref{Sol2}, and \eqref{Sol3} is violated infinitely often.
Our method in fact allows us to determine information on $S_\alpha(x)$ for any $\alpha\in[0,1]$.
For this, we define
\begin{equation}\label{eqnScriptS}
\mathcal{S}_\alpha(x) = \begin{cases}
\displaystyle S_\alpha(x), & \textrm{if $0\leq\alpha<1/2$ or $\alpha=1$},\\
\displaystyle S_\alpha(x) + \frac{1+\sqrt{2}}{2\zeta(1/2)}\log x, & \textrm{if $\alpha=1/2$},\\[8pt]
\displaystyle S_\alpha(x) + \frac{(1+2^{-\alpha})\zeta(2\alpha)}{\zeta(\alpha)}, & \textrm{if $1/2<\alpha<1$.}
\end{cases}
\end{equation}
We determine bounds on oscillations for $\mathcal{S}_\alpha(x)$ for all $\alpha\in[0,1]$ in the following theorem.

\begin{theorem}\label{thmSAsymp}
For each $\alpha\in[0,1/2)\cup(1/2,1]$, each of the following inequalities is satisfied for infinitely many positive integers $x$:
\begin{equation}\label{eqnSaGen}
\begin{split}
\mathcal{S}_\alpha(x)x^{\alpha - \frac{1}{2}} &< \frac{1+\sqrt{2}}{(2\alpha-1)\zeta(1/2)} - 1.6725193,\\
\mathcal{S}_\alpha(x)x^{\alpha - \frac{1}{2}} &> \frac{1+\sqrt{2}}{(2\alpha-1)\zeta(1/2)} + 1.6725193.
\end{split}
\end{equation}
In particular, each of the following holds for infinitely many positive integers $x$:
\begin{equation}\label{eqnSAsymp}
\begin{split}
\mathcal{S}_0(x)x^{-1/2} < -0.019349, &\quad \mathcal{S}_0(x)x^{-1/2} > 3.32568,\\
\mathcal{S}_{1/4}(x)x^{-1/4} < 1.63369, &\quad \mathcal{S}_{1/4}(x)x^{-1/4} > 4.97900,\\
\mathcal{S}_{3/4}(x)x^{1/4} < -4.97900, &\quad \mathcal{S}_{3/4}(x)x^{1/4} > -1.63369,\\
\mathcal{S}_1(x)x^{1/2} < -3.32568, &\quad \mathcal{S}_1(x)x^{1/2} > 0.019349.
\end{split}
\end{equation}
In addition, each of the following inequalities is satisfied for infinitely many positive integers $x$:
\begin{equation}\label{eqnSaHalf}
\mathcal{S}_{1/2}(x) < -3.27438, \quad \mathcal{S}_{1/2}(x) > 0.071048.
\end{equation}
\end{theorem}
% 1.67271899813289 -3.27438947510773 0.07104852115805

In this article we investigate another conjecture made by Sun as well in \cite{Sun}.
Let
\begin{equation}\label{eqnW}
W(x) = \sum_{n\leq x} (-2)^{\Omega(n)}.
\end{equation}
The similar function $W_+(x) = \sum_{n\leq x} 2^{\Omega(n)}$ was considered by Grosswald \cite{Grosswald56} and by Bateman \cite{Gordon}: in particular, Grosswald proved that $\abs{W_+(x)} \ll x\log^2 x$.
In \cite[Conj.\ 1.1]{Sun} Sun conjectured that a smaller and explicit bound holds for $W(x)$:
\begin{equation}\label{Sol4}
\abs{W(x)} < x  \textrm{\;\ for $x\geq 3078$}.
\end{equation}
We investigate this function here, and verify that \eqref{Sol4} holds for $x\leq2.5\cdot10^{14}$.

This article is organized in the following way.
Section~\ref{secWeakInd} describes the notion of weak independence of a set of real numbers, and its use in establishing lower bounds on the oscillation of sums of certain arithmetic functions.
Section~\ref{secS} presents the analysis of the functions $S_\alpha(x)$, and establishes Theorem~\ref{thmSAsymp}.
Following this, Section~\ref{secH} considers the family $H_\alpha(x)$, and proves Theorem~\ref{thmHosc}.
Details on the computations that establish Theorem~\ref{thmSComps} and related facts, along with information on calculations for $W(x)$, are described in Section~\ref{secComps}.
Finally, in Section~\ref{secMoreConj} we touch on some other conjectures of Sun from \cite{Sun}.

\section{Weak independence and oscillations}\label{secWeakInd}

The work of Grosswald \cite{Grosswald67,Grosswald72}, Bateman et al.\ \cite{Bateman}, Diamond \cite{Diamond}, Anderson and Stark \cite{AndStark}, and others establishes connections between bounding oscillations in sums of certain arithmetic functions and the existence of integral linear relations among the ordinates of nontrivial zeros of $\zeta(s)$, where the linear relations have bounded coefficients.
For example, in 1971 Bateman et al.\ \cite{Bateman} showed that if the Mertens function \eqref{eqnMertens} satisfied $\abs{M(x)}\leq\sqrt{x}$ for all $x>0$ then there would exist infinitely many linear relations among the nontrivial zeros of $\zeta(s)$ in the upper half plane with maximum coefficient $2$. They showed as well that the ordinates of the first $20$ zeros of $\zeta(s)$ on the critical line are $1$-independent, that is, there exist no nontrivial linear relations among these values using coefficients $\{-1,0,1\}$.
In a more recent study of oscillations in the Mertens function \cite{Best}, it was shown that the first $500$ zeros of $\zeta(s)$ are $10^{5}$-independent.

Here, in order to exhibit large oscillations in sums we require a particular form of weak linear independence for certain zeros of $\zeta(s)$.
We assume the Riemann hypothesis and the simplicity of its zeros here and throughout this paper, as unbounded oscillations would follow anyway if either of these conditions did not hold.
Let $\Gamma$ denote a set of positive real numbers, and for $T>1$ let $\Gamma'=\Gamma'(T)$ denote a subset of $\Gamma \cap [0,T]$.
For each $\gamma\in\Gamma'$, let $N_\gamma$ denote a positive integer.
We say $\Gamma'$ is \textit{$\{N_\gamma\}$-independent} in $\Gamma \cap [0, T]$ if two conditions hold.
First, whenever
\begin{equation*}\label{ind1}
\sum_{\gamma\in\Gamma'} c_{\gamma} \gamma = 0, \textrm{\;\ with $\abs{c_{\gamma}} \leq N_\gamma,\; c_{\gamma} \in \mathbb{Z}$},
\end{equation*}
then necessarily all $c_{\gamma} = 0$.
Second, for any $\gamma^{*}\in \Gamma \cap [0, T]$, if
\begin{equation*}\label{ind2}
\sum_{\gamma\in\Gamma'} c_{\gamma} \gamma = \gamma^{*}, \textrm{\;\ with $\abs{c_{\gamma}} \leq N_\gamma,\; c_{\gamma} \in \mathbb{Z}$},
\end{equation*}
then $\gamma^{*} \in \Gamma'$, $c_{\gamma*} = 1$, and all other $c_{\gamma} = 0$.
That is, the only linear relations with small coefficients here are the trivial ones.
In addition, if $N$ is a positive integer with the property that each $N_\gamma \geq N$ with $\{N_\gamma\}$ as above, then we say that $\Gamma'$ is \textit{$N$-independent} in $\Gamma \cap [0, T]$.

Anderson and Stark \cite{AndStark} established a means for bounding oscillations by using weak independence.
We summarize a particular result of theirs: suppose $g(u)$ is a piecewise continuous, real function, bounded on finite intervals, with Laplace transform $G(s)$,
\[
G(s) = \int_0^\infty g(u) e^{-su}\,du.
\]
Suppose also that $G(s)$ is absolutely convergent in $\sigma>\sigma_0$ for some $\sigma_0$, and can be analytically continued back to $\sigma\geq0$, except for simple poles occurring at $\pm i\gamma$ for $\gamma\in\Gamma$, for a particular set $\Gamma$ of positive real numbers, and possibly a simple pole at $0$ as well.
If $\Gamma'\subseteq\Gamma$ is $\{N_\gamma\}$-independent in $\Gamma\cap[0,T]$, then
\begin{equation}\label{eqnAndStark}
\begin{split}
\liminf_{u\to\infty} g(u) &\leq \Res(G,0) - 2\sum_{\gamma\in\Gamma'} \frac{N_\gamma}{N_\gamma+1} k_T(\gamma) \abs{\Res(G,i\gamma)},\\
\limsup_{u\to\infty} g(u) &\geq \Res(G,0) + 2\sum_{\gamma\in\Gamma'} \frac{N_\gamma}{N_\gamma+1} k_T(\gamma) \abs{\Res(G,i\gamma)},
\end{split}
\end{equation}
where $k_T(x)$ denotes an admissible weight function, which must be even, nonnegative, supported on $[-T,T]$, have the value $1$ at $x=0$, and be the Fourier transform of a nonnegative function.
While the Fej\'er kernel ($k_T(x) = 1-\abs{x}/T$ for $\abs{x}\leq T$ and $0$ otherwise) is admissible, we employ the kernel of Jurkat and Peyerimhoff \cite{JP} here, defined by
\begin{equation}\label{eqnJPkernel}
k_T(x) = \begin{cases}
\displaystyle \left(1-\frac{\abs{x}}{T}\right)\cos\left(\frac{\pi x}{T}\right) + \frac{1}{\pi}\sin\left(\frac{\pi\abs{x}}{T}\right), & \textrm{if $\abs{x}\leq T$},\\
\displaystyle 0, &  \textrm{if $\abs{x} > T$}.
\end{cases}
\end{equation}
This kernel provides more weight to values of $x$ near the origin compared to the Fej\'er kernel (and correspondingly less weight to values farther away), and this is often advantageous.

The result \eqref{eqnAndStark} then provides a weak version of Ingham's theorem for the case of weak independence.
Ingham established a result similar to \eqref{eqnAndStark}, where linear independence over the rationals for $\Gamma$ was required, the fractions $N_\gamma/(N_\gamma+1)$ were naturally omitted, and $\Gamma'$ was replaced by $\Gamma\cap[0,T]$.

The LLL lattice reduction algorithm \cite{LLL}, combined with some additional linear algebra, may be employed to establish weak independence of a finite set $\Gamma'$ of positive real numbers, relative to a parent set $\Gamma$ and a real parameter $T$.
In this article, we are able to employ some weak independence results established in prior work, so we omit a detailed description of the algorithm, and refer the reader to the following sources.
The article \cite{MosTru} provides a detailed explanation of Ingham's method, extended to the family of functions $L_\alpha(x)$ from \eqref{eqnL}, including connections with the Riemann hypothesis, the simplicity of the zeros of the zeta function, and the linear independence question.
The publications \cite{Best,Olive,MT2017,Dickens} describe the method of weak independence, and computational strategies for establishing this property.
The first of these treats the Mertens function $M(x)$, the third deals with the family $L_\alpha(x)$, and the last one considers $H_0(x)$ from \eqref{eqnH}.
In particular, the paper \cite{MT2017} contains substantial exposition on this method, including for example a proof of \eqref{eqnAndStark}.

\section{Oscillations in $S_\alpha(x)$}\label{secS}

Since $S_\alpha(x)$ from \eqref{eqnSalpha} is rather similar to P\'olya's function $L_\alpha(x)$, one might begin by employing some analysis similar to that of \cite{MosTru} and \cite{MT2017}.
Indeed, when $\alpha=0$ there is a simple relationship connecting these two functions:
\begin{align*}
S_0(2n) &= \sum_{\substack{k\leq2n\\2\mid k}} (-1)^{\Omega(k)} - \sum_{\substack{k\leq2n\\2\nmid k}} (-1)^{\Omega(k)}\\
&= -\sum_{j\leq n} (-1)^{\Omega(j)} - \sum_{k\leq 2n} (-1)^{\Omega(k)} + \sum_{\substack{k\leq2n\\2\mid k}} (-1)^{\Omega(k)}\\
&= -L_0(2n) - 2L_0(n).
\end{align*}
Thus, we might expect $S_\alpha(x)$ to exhibit a bias with sign opposite to that of $L_\alpha(x)$, and we might hope that methods that established large oscillations in $L_\alpha(x)$ might also produce similarly good results for $S_\alpha(x)$.
Indeed, this turns out to be the case, as we show below.

Let $Y(s)$ denote the Dirichlet series
\[
Y(s) = \sum_{n=1}^{\infty} \frac{(-1)^{n - \Omega(n)}}{n^{s}}.
\]
If $\sigma>1$, then it is straightforward to establish that
\begin{equation*}\label{pedal}
Y(s) = -(1 + 2^{1-s}) \frac{\zeta(2s)}{\zeta(s)}.
\end{equation*}
This was also given by Sun \cite{Sun}.
Standard results (see, e.g., \cite{MosTru}) show that if either the Riemann hypothesis is false or a zero of $\zeta(s)$ is not simple, then
\begin{equation*}\label{keys}
\limsup_{x\rightarrow \infty} \frac{S_\alpha(x)}{x^{\frac{1}{2}-\alpha}} = \infty \textrm{\;\ and\ \;} \liminf_{x\rightarrow \infty} \frac{S_\alpha(x)}{x^{\frac{1}{2}-\alpha}} = -\infty.
\end{equation*}
Hence we may assume the Riemann hypothesis and the simplicity of the zeros in order to exhibit large oscillations.
We follow the same procedure as in \cite{MosTru}.
With $\mathcal{S}_\alpha(x)$ as in \eqref{eqnScriptS}, we define
\[
B_\alpha(u) = \mathcal{S}_\alpha(e^u) e^{(\alpha-\frac{1}{2})u}.
\]
The Laplace transform of $B_\alpha(u)$ is
\[
F_\alpha(s) = \int_0^\infty B_\alpha(u) e^{-su}\,du,
\]
and as in \cite{MosTru} we compute
\begin{equation}\label{eqnLTS}
F_\alpha(s) = \begin{cases}
\displaystyle f_\alpha(s), & \textrm{if $0\leq\alpha<1/2$, or $\alpha=1$},\\
\displaystyle f_\alpha(s) + \frac{1+\sqrt{2}}{2\zeta(1/2)s^2}, & \textrm{if $\alpha=1/2$},\\[8pt]
\displaystyle f_\alpha(s) + \frac{(1+2^{-\alpha})\zeta(2\alpha)}{\zeta(\alpha)(s-\alpha+1/2)}, & \textrm{if $1/2<\alpha<1$},
\end{cases}
\end{equation}
where
\begin{equation}\label{eqnSf}
f_\alpha(s) = \frac{Y(s+1/2)}{s-\alpha+1/2} = -\frac{(1+2^{\frac{1}{2}-s})\zeta(2s+1)}{(s-\alpha+1/2)\zeta(s+1/2)}.
\end{equation}
The adjustments made to $S_\alpha(x)$ when creating $\mathcal{S}_\alpha(x)$ in \eqref{eqnScriptS} are engineered so that the function $F_\alpha(s)$ in \eqref{eqnLTS} is analytic in the open half-plane $\{s= \sigma + it \in \mathbb{C}:\, \sigma > 0\}$ with only simple poles on the imaginary axis.
In particular, the adjustment at $\alpha=1/2$ removes the pole of order $2$ at $s=0$ in \eqref{eqnSf}, and the adjustment for $1/2<\alpha<1$ removes the simple pole on the real axis at $s=\alpha-1/2$.
This prepares us to employ the result of Anderson and Stark and the machinery of weak independence.

For $s=0$, we compute
\begin{equation}\label{alfalfa}
\Res(F_\alpha,0) = \begin{cases}
\displaystyle -\frac{1+\sqrt{2}}{(2\alpha-1)\zeta(1/2)}, & \textrm{if $\alpha\in[0,1/2)\cup(1/2,1]$},\\[8pt]
\displaystyle \frac{1+\sqrt{2}}{\zeta(1/2)}\left(\frac{\log 2}{2+\sqrt{2}}+\frac{\zeta'(1/2)}{2\zeta(1/2)}-\gamma_{0}\right), & \textrm{if $\alpha=1/2$,}
\end{cases}
\end{equation}
where $\gamma_{0}=0.57721\ldots$ denotes Euler's constant.
For $s=i\gamma_n$, where $\rho_n=\frac{1}{2}+i\gamma_n$ is the $n$th zero of the zeta function on the critical line in the upper half plane, we have
\begin{equation}\label{sorghum}
\Res(F_\alpha, i\gamma_n) = -\frac{(1+2^{\overline{\rho}_n})\zeta(2\rho_n)}{(\rho_n-\alpha)\zeta'(\rho_n)}.
\end{equation}

If $\Gamma'\subseteq\Gamma\cap[0,T]$ and $\Gamma'$ is $N$-independent in $\Gamma\cap[0,T]$, then using the result of Anderson and Stark \eqref{eqnAndStark} (where $\sigma_0=1/2$) we conclude that
\begin{equation}\label{wheat}
\begin{split}
\liminf_{u\to\infty} B_\alpha(u) &\leq \Res(F_\alpha,0) - \frac{2N}{N+1}\sum_{\gamma\in\Gamma'} k_T(\gamma)\abs{\Res(F_\alpha, i\gamma)},\\
\limsup_{u\to\infty} B_\alpha(u) &\geq \Res(F_\alpha,0) + \frac{2N}{N+1}\sum_{\gamma\in\Gamma'} k_T(\gamma)\abs{\Res(F_\alpha, i\gamma)}.
\end{split}
\end{equation}
In particular, when $\alpha=0$, this produces (writing $\rho=\frac{1}{2}+i\gamma$)
\begin{equation}\label{eqnS0high}
\frac{S_0(e^u)}{e^{u/2}}  > -\frac{1 + \sqrt{2}}{\zeta(1/2)} + \sum_{\gamma\in\Gamma'} k_T(\gamma)\abs{\frac{(1 + 2^{\overline{\rho}})\zeta(2\rho)}{\rho \zeta'(\rho)}} - \epsilon
\end{equation}
for an infinite sequence of $u\to\infty$, and likewise
\begin{equation}\label{eqnS0low}
\frac{S_0(e^u)}{e^{u/2}} < - \frac{1 + \sqrt{2}}{\zeta(1/2)} - \sum_{\gamma\in\Gamma'} k_T(\gamma)\abs{\frac{(1 + 2^{\overline{\rho}})\zeta(2\rho)}{\rho \zeta'(\rho)}} + \epsilon
\end{equation}
for another infinite sequence of $u\to\infty$.
Note that $-(1 + \sqrt{2})/\zeta(1/2) = 1.653\ldots$\,, so that, in a loose sense, one expects $S_{0}(x)$ to be biased towards positive values.
Similarly, $S_1(e^u)e^{u/2}$ exhibits the same oscillation bounds about $(1 + \sqrt{2})/\zeta(1/2)=-1.653\ldots$\,, since $\abs{\rho_n-1}=\abs{\rho_n}$.
We note that the conjectures \eqref{Sol2} and \eqref{Sol3} are the same as stating that the magnitude of the oscillations here never exceeds approximately $0.653$, since $1.653 - 0.653 = 1$ and $1.653 + 0.653 \approx 2.3$.

We may now complete the proof of Theorem~\ref{thmSAsymp} by employing a result on weak independence.

\begin{proof}[Proof of Theorem~\ref{thmSAsymp}]
In \cite{MT2017}, we determined a particular set $\Gamma'\subseteq\Gamma$ having $\abs{\Gamma'}=250$ with the property that $\Gamma'$ is $3100$-independent in $[0,\gamma_{3701}-10^{-10}]$.
Employing this set in \eqref{eqnS0high} and \eqref{eqnS0low}, and using the kernel of Jurkat and Peyerimhoff \eqref{eqnJPkernel}, produces oscillations with magnitude $1.6725193\ldots$ in each direction.
Since $\abs{\frac{1}{2}+i\gamma}\geq\abs{\frac{1}{2}+i\gamma-\alpha}\geq\abs{\gamma}$ for $0\leq\alpha\leq1$, the first statement follows.
The subsequent statements for $\alpha=0$ and $\alpha=1$ in \eqref{eqnSAsymp} recapitulate the inequalities in \eqref{eqnSaGen} for these cases.
However, the statements for $\alpha=1/4$ and $\alpha=3/4$ employ the oscillation bounds from \eqref{wheat} computed as $1.6726690$ for these cases.
Likewise, the final statement \eqref{eqnSaHalf}, regarding the case $\alpha=1/2$, employs the computed oscillation bound of $1.67271899$ for this case (which is in fact maximal over $0\leq\alpha\leq1$), together with the appropriate entries from \eqref{eqnScriptS} and \eqref{alfalfa}.
\end{proof}

We remark that it is possible to improve the bounds in Theorem~\ref{thmSAsymp} with a new computation, using zeros selected to maximize contributions according to the residues \eqref{sorghum} in this problem, rather than employing the set of residues that were selected for the functions $L_\alpha(x)$.
However, we would not witness a substantial gain with a  calculation of approximately the same size: we estimate gaining $0.015$ in the oscillation bound in each direction from a computation with approximately $n=240$ zeros.
This computation would require several months of core time, so with a high cost relative to a small benefit we opted against a new calculation here.

Figures~\ref{figS0sampling} and~\ref{figS1sampling} illustrate the actual oscillations exhibited by $S_0(e^u)e^{-u/2}$ and $S_1(e^u)e^{u/2}$ respectively over $24\leq u\leq 34.25\ldots$\,, along with their respective center line at $u=\pm(1+\sqrt{2})/\zeta(1/2)$.
Also, note that when $\alpha=1/2$ the oscillations of $S_{1/2}(e^u)$ center on the line
\begin{equation}\label{eqnS2center}
-\frac{1+\sqrt{2}}{2\zeta(1/2)}u + \frac{1+\sqrt{2}}{\zeta(1/2)}\left(\frac{\log 2}{2+\sqrt{2}}+\frac{\zeta'(1/2)}{2\zeta(1/2)}-\gamma_{0}\right) \approx 0.826585u -1.60167.
\end{equation}
Figure~\ref{figSQ2plot} displays the sampled values of $S_{1/2}(u)$ over $24\leq u\leq30$, along with this line.
% {0.8265847600, -1.601670477}

\section{Oscillations in $H_\alpha(x)$}\label{secH}

In order to prove Theorem~\ref{thmHosc}, we generalize our analysis in \cite{Dickens}, in a manner analogous to the prior section and our study of the behavior $L_\alpha(x)$ from \cite{MosTru} and \cite{MT2017}.
In \cite{Dickens}, for the case $\alpha=0$ we computed the Laplace transform of $H_0(e^u)/e^{u/2}$: under the Riemann hypothesis, this function is analytic in the half plane $\sigma>0$.
For the general case, we proceed in a similar manner, except that when $\alpha\in(1/2,1)$ one must incorporate an adjustment to prevent the appearance of a pole on the positive real axis in the corresponding Laplace transform.
In essence one must account for the value of $h(\alpha)$ for $1/2<\alpha<1$, with $h(\alpha)$ defined in this interval by \eqref{chalk}.
We made precisely the analogous adjustment in the analysis of $L_\alpha(x)$ in \cite{MosTru} and \cite{MT2017}.
Using the function $\mathcal{H}_\alpha$ from \eqref{eqnHcal}, for $u\geq 0$ we set $A_\alpha(u)$ to be a suitable scaling of $\mathcal{H}_\alpha(e^u)$:
\begin{equation*}\label{stalk}
A_{\alpha}(u) = \mathcal{H}_{\alpha}(e^{u}) e^{(\alpha - \frac{1}{2}) u}.
\end{equation*}
Next, we set
\begin{equation}\label{hawk}
h_{\alpha}(s) = \frac{h(s+\frac{1}{2})}{s-\alpha + \frac{1}{2}}
\end{equation}
and
\begin{equation*}\label{baulk}
G_{\alpha}(s) =
\begin{cases} 
\displaystyle h_{\alpha}(s), & \textrm{if $0 \leq \alpha \leq \frac{1}{2}$},\\
\displaystyle h_{\alpha}(s) - \frac{h(\alpha)}{s- \alpha + \frac{1}{2}}, & \textrm{if $\frac{1}{2} <\alpha \leq 1$}.
\end{cases}
\end{equation*}
Using an argument very similar to that employed in \cite{MosTru} for $L_\alpha(x)$, we find that for all $\alpha\in[0,1]$, the function $G_\alpha(s)$ is the Laplace transform of $A_\alpha(s)$:
\begin{equation}\label{caulk}
G_{\alpha}(s) = \int_{0}^{\infty} A_{\alpha}(u) e^{-su}\, du,
\end{equation}
and, under the Riemann hypothesis, the integral in \eqref{caulk} converges for all $\sigma >0$.
One slight difference arises here: when $\alpha=1/2$ there is a pole of $h_{\alpha}(s)$ in \eqref{hawk} when $s=0$.
However, this is canceled by the zero of $h(1/2)$, which itself comes from the pole of $\zeta(2s)$ in \eqref{chalk}.
As such, unlike the case of $L_{1/2}(x)$ or $S_{1/2}(x)$, there is no dramatic bias in $H_{1/2}(x)$.

It is now straightforward to calculate the residue of $G_{\alpha}(s)$ at each pole along $\sigma = 0$.
We have
\begin{equation}\label{dork}
\Res(G_\alpha, 0) =
\begin{cases}  0, & \textrm{if $0 \leq \alpha \leq \frac{1}{2}$},\\
h(\alpha), & \textrm{if $\frac{1}{2} <\alpha \leq 1$}.
\end{cases}
\end{equation}
From \eqref{dork} and the fact that $h(1)=0$, we expect no bias in the sum $H_{\alpha}(x)$ for $0\leq \alpha \leq 1/2$ and for $\alpha =1$.
Over $1/2 < \alpha <1$, we expect a bias in this function whenever $h(\alpha)\neq0$.
Figure~\ref{fighplot} exhibits $h(\alpha)$ over $1/2\leq\alpha\leq1$.
It is interesting that the sign of the bias changes over this range.
For $1/2<\alpha<\alpha^*\approx0.63093$, our plot indicates that we expect $H_\alpha(x)$ to exhibit a \textit{negative} bias, with maximal negative bias occurring near $\alpha_1=0.55336$, where $h(\alpha_1)=-0.0950579\ldots$\,.
That is, the values of $H_{\alpha_1}(x)$ will oscillate relative to the base curve $h(\alpha_1)e^{(\alpha_1-\frac{1}{2})u}$.
For $\alpha^*<\alpha<1$, we expect the sum $H_\alpha(x)$ to display a \textit{positive} bias, with maximal bias occurring near $\alpha_2=0.73587$, where $h(\alpha_2)=0.0804324\ldots$\,.
Figure~\ref{figHbias} displays $H_{\alpha}(e^u) e^{\alpha-1/2}$ for two values of $\alpha$ near these points of maximal bias.
In these plots, the bias specified by $h(\alpha)$ was not removed as in \eqref{eqnHcal}, so in each of these graphs the oscillations center on an exponential function.
In part (b) of this figure, where $\alpha=3/4$, the bias is quite evident, since the centering curve is $e^{u/4}$.
However, in part (a) it is not so obvious, since the guiding exponential $e^{u/20}$ remains rather small over the plotted interval.

\begin{figure}[tb]
\begin{center}
\includegraphics[width=2.5in]{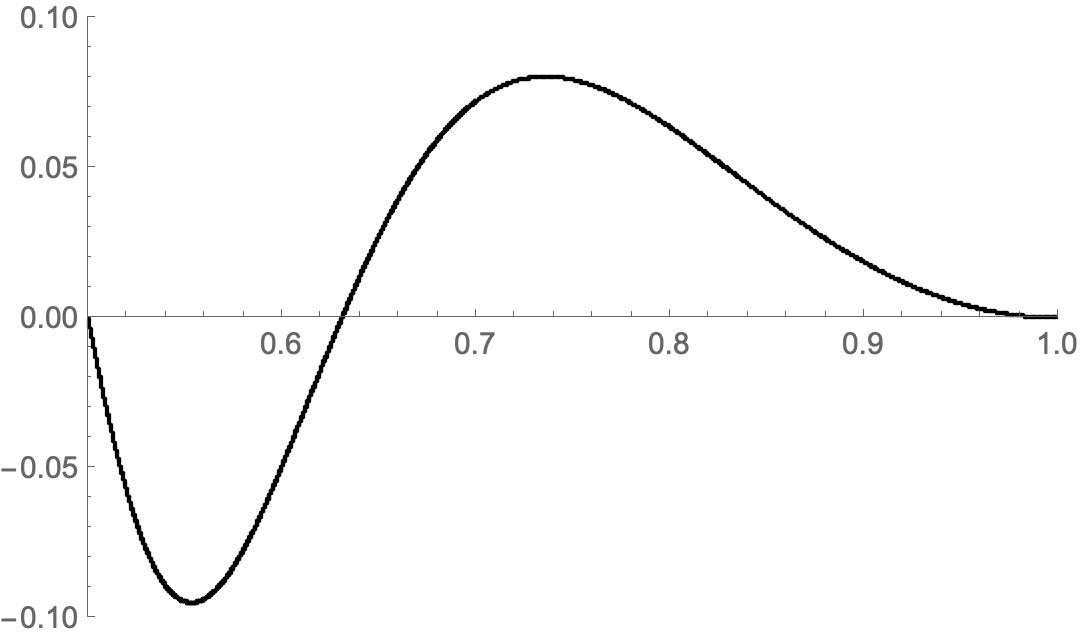}
\end{center}
\caption{$h(\alpha)$ for $\frac{1}{2}\leq\alpha\leq1$.}\label{fighplot}
\end{figure}

We may now establish Theorem~\ref{thmHosc} by computing the relevant residues and employing a particular set of zeros of the zeta function known to possess a weak independence property.

\begin{proof}[Proof of Theorem~\ref{thmHosc}]
In \cite{Dickens}, we computed the residue of $G_0(s)$ at each of its simple poles on the imaginary axis, so at each point $s=i\gamma_n$ where $\rho_n=\frac{1}{2}+\gamma_n$ denotes the $n$th zero of the Riemann zeta function on the critical line in the upper half plane.
This derivation generalizes in a natural way and yields
\[
\Res(G_\alpha, \gamma_n) = \frac{F_6(\rho_n)}{(\rho_n-\alpha)\zeta'(\rho_n)Z_6(\rho_n)}.
\]
The same article employed a particular set $\Gamma'\subseteq[0,T]$ of ordinates of $239$ zeros of the zeta function on the critical line, with $T=\gamma_{3701}-10^{-10}$, to establish large oscillations for $\mathcal{H}_0$.
By proving that $\Gamma'$ is $N$-independent in $\Gamma\cap[0,T]$ with $N=3950$, we obtained Theorem~\ref{thmHosc} from \eqref{eqnAndStark} in the case $\alpha=0$, since
\begin{equation}\label{eqnHLB0}
\frac{2N}{N+1}\sum_{\gamma\in\Gamma'} k_T(\gamma) \abs{\Res(G_0, \gamma)} > 1.700144.
\end{equation}
Theorem~\ref{thmHosc} follows almost immediately.
Since $\abs{\frac{1}{2}+i\gamma}\geq\abs{\frac{1}{2}+i\gamma-\alpha}\geq\abs{\gamma}$ for $0\leq\alpha\leq1$, certainly
\[
\abs{\Res(G_0, \gamma)} \leq \abs{\Res(G_\alpha, \gamma)} \leq \abs{\Res(G_{1/2}, \gamma)}
\]
over this range, and thus the bound \eqref{eqnHLB0} holds as well when $G_0$ is replaced by $G_\alpha$, for any $\alpha\in[0,1]$.
\end{proof}

We remark that since $\rho_n-(1-\alpha)=\overline{\alpha-\rho_n}$, we see that $\abs{\rho_n-\alpha}=\abs{\rho_n-(1-\alpha)}$, so the bound for $G_{1-\alpha}$ is in fact exactly the same as that for $G_\alpha$ for $\alpha\in[0,1]$.
In particular, the bound for $\alpha=1$, which was of particular interest in \cite{Sun}, matches the one for $\alpha=0$.
We add also that our method establishes bounds on oscillations for $0<\alpha<1$ that are only slightly stronger than \eqref{eqnHLB0}.
In the most beneficial case (when $\alpha=1/2$), we obtain
\[
\frac{2N}{N+1}\sum_{\gamma\in\Gamma'} k_T(\gamma) \abs{\Res(G_{1/2}, \gamma)} > 1.700282.
\]

Figure~\ref{figAH} displays plots of $A_\alpha(u)=\mathcal{H}_\alpha(u)e^{(\alpha-\frac{1}{2})u}$ for a number of values of $\alpha$.
These plots were obtained from sampling the values of these functions $H_\alpha(x)$ for $x\leq e^{30}$.
The graphs also display the axes where the oscillations are centered.
In part (d) of this figure, where $\alpha=3/4$, the axis is the horizontal line at $h(3/4)=0.0793843\ldots$\,; for the others it is the horizontal axis.
% 0.079384313103580659706

\section{Computations}\label{secComps}

We completed a number of computations to sample the values of $S_\alpha(x)$ and $H_\alpha(x)$ for several choices of the parameter $\alpha$, as well as $W(x)$ from \eqref{eqnW}.
Our method is similar to the one developed in \cite{Borwein} for the investigation of $L_0(x)$ and $L_1(x)$; a similar method was employed in \cite{MT2017} for $L_\alpha(x)$ and in \cite{Dickens} for $H_0(x)$.
We summarize the method used here for obtaining values of the functions $S_\alpha(x)$, in order to provide information for Theorem~\ref{thmSComps}.

By using Perron's formula and the procedure of \cite{MosTru}, we find that
\begin{equation}\label{teak}
\frac{S_0(e^u)}{e^{u/2}} = -\frac{1 + \sqrt{2}}{\zeta(1/2)} - \sum_{\abs{\gamma_n}\leq T} \frac{(1 + 2^{1/2 - i\gamma_n})\zeta(2\rho_n) e^{i\gamma_n u}}{\rho_n\zeta'(\rho_n)} + E(u, T),
\end{equation}
where the error term $E(u,T)$ satisfies $E(u,T)\to0$ as $u\to\infty$.
We use this to estimate the behavior of $S_0(e^u)$, by setting $E(u,T)=0$ in \eqref{teak} and calculating the expression on the right with $T$ set to certain values, and then plotting the results.
When $T=3000$, so with $2469$ zeros of the zeta function in the upper half plane, we detected a potential crossing point for the lower bound in \eqref{Sol2}---see Figure~\ref{figS0cross}(a).
This suggests that $S_0(x)/\sqrt{x}$ may dip below $1$ near $x=\exp(34.186)\approx 7.027\cdot10^{14}$, and similarly $S_1(x)\sqrt{x}$ may cross the lower threshold of $-2.3$ from \eqref{Sol3} near the same location.
These values seemed within range of a distributed computation, so we aimed to calculate $S_0(x)$ and $S_1(x)$ for $x\leq 7.5\cdot 10^{14}$.

\begin{figure}[tb]
\begin{center}
\begin{tabular}{cc}
\includegraphics[width=2.25in]{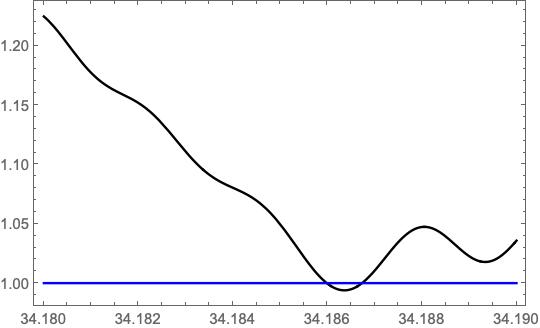} &
\includegraphics[width=2.25in]{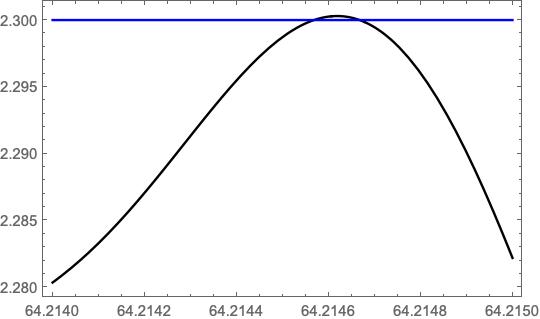}\\
(a) $[34.18, 34.19]$ ($T=3000$) & (b) $[64.214, 64.215]$ ($T=5200$)
\end{tabular}
\end{center}
\caption{Estimating $S_0(e^u)/e^{u/2}$ over particular intervals using \eqref{teak}.}\label{figS0cross}
\end{figure}

\begin{figure}[tb]
\begin{center}
\includegraphics[width=3.5in]{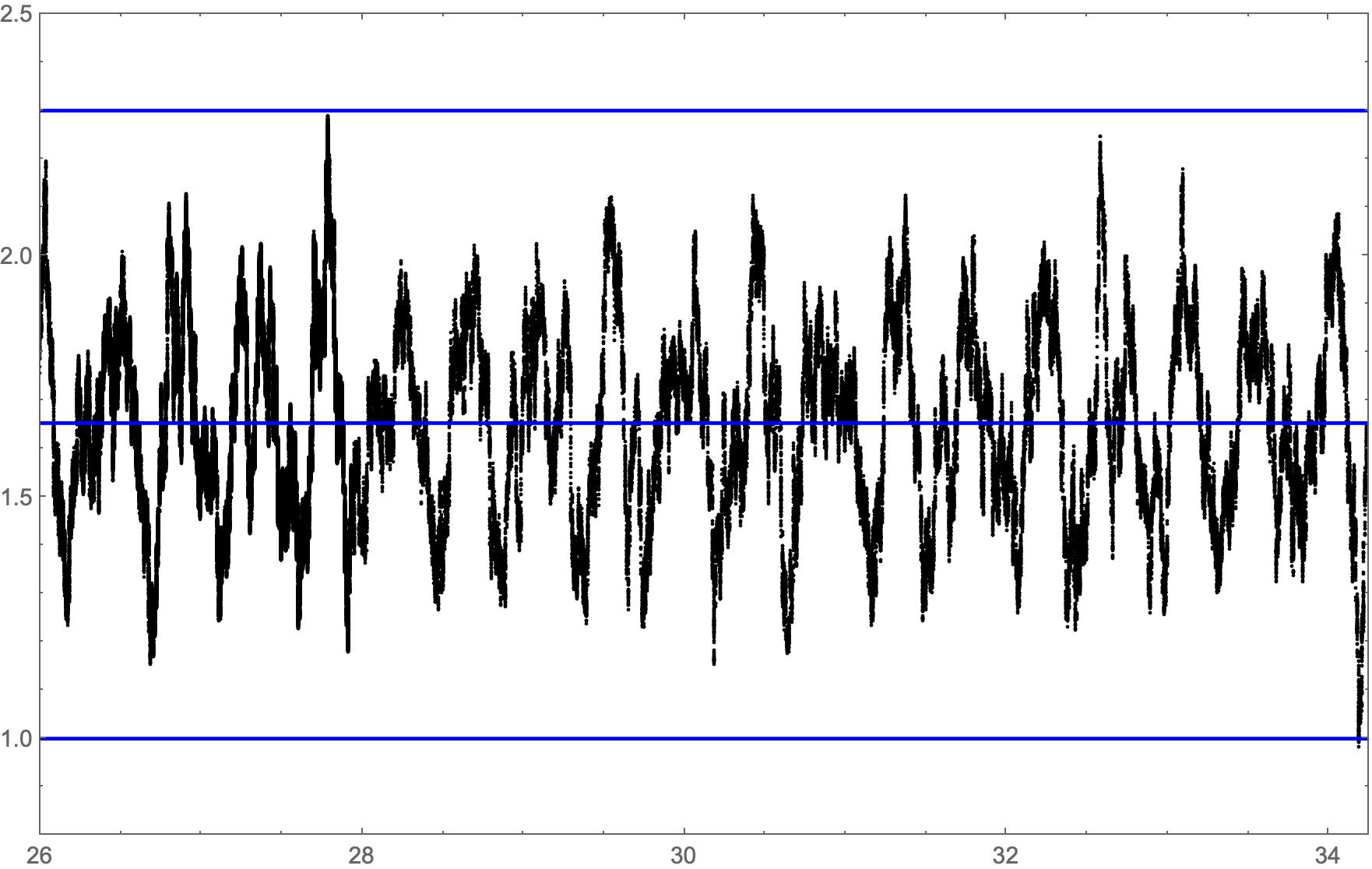}
\caption{Sampled values for $S_0(e^u)/e^{u/2}$, oscillating about $-(1+\sqrt{2})/\zeta(1/2)$.}\label{figS0sampling}
\end{center}
\end{figure}

\begin{figure}[tb]
\begin{center}
\includegraphics[width=3.5in]{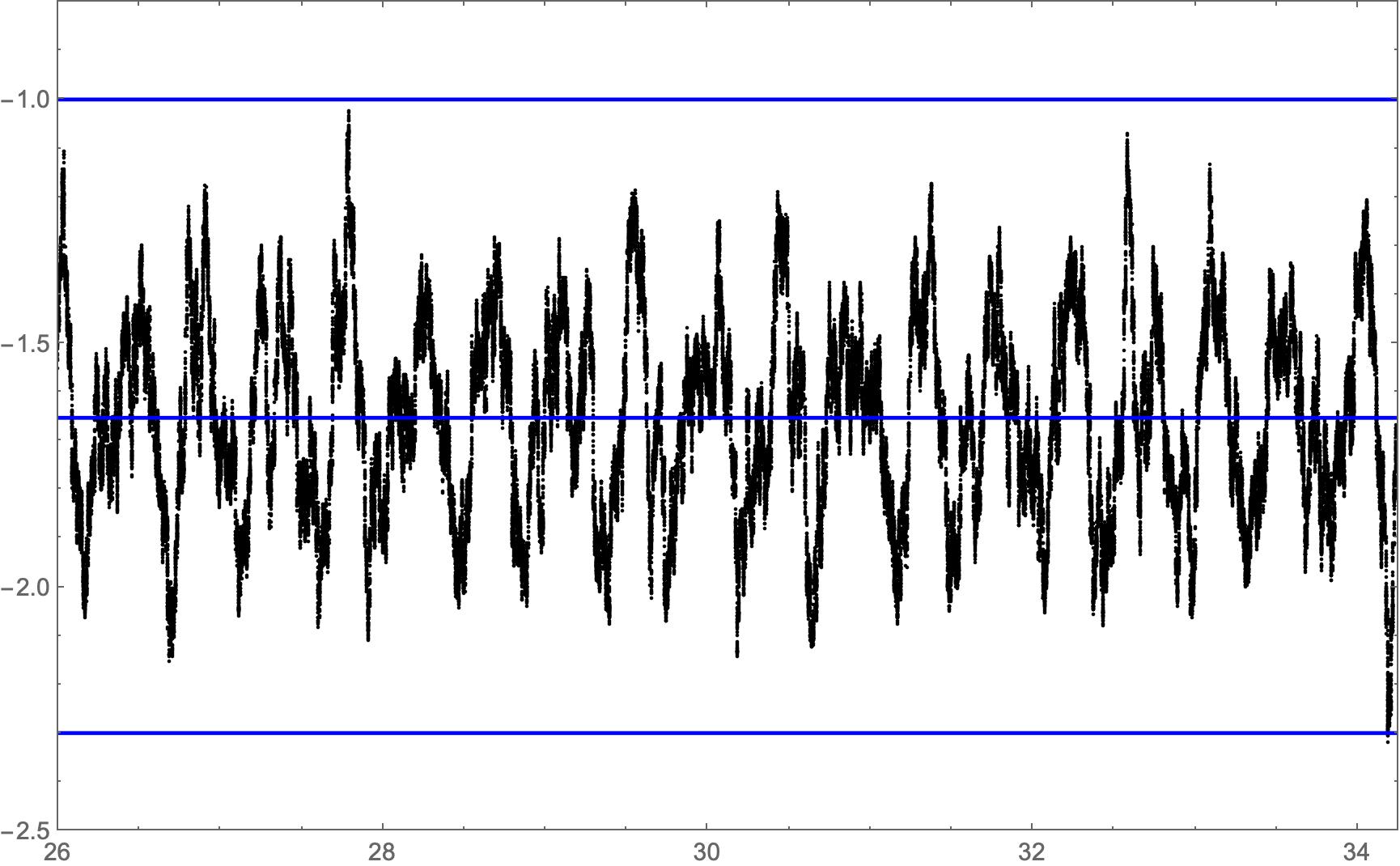}
\caption{Sampled values for $S_1(e^u)e^{u/2}$, oscillating about $(1+\sqrt{2})/\zeta(1/2)$.}\label{figS1sampling}
\end{center}
\end{figure}

To calculate values for $S_0(x)$, we first created a large table to store the parity of the values of $n-\Omega(n)$ for all positive integers $n$ with $\gcd(n,30)=1$ and $n\leq M_S=5\cdot 10^{10}$.
This required about $1.55$ gigabytes of storage.
We employed a boot-strapping strategy to create this table: once the values up to a bound $y$ have been computed, we can then sieve the interval $[y+1, 2y]$ using the primes $p\leq\sqrt{2y}$ to determine the parity of the total number of divisors for each qualifying integer in this interval.
Here we use the existing table for $[1,y]$ whenever a remaining cofactor lands in this interval, once powers of $2$, $3$, and $5$ are properly accounted for.
After the full table to $M_S$ has been constructed, we launch a distributed calculation, with each processor employing a copy of the table to assist with sieving its own portion of the targeted range.
Each process handled an interval of size $8\cdot10^{10}$, usually breaking this up into $3200$ subintervals of size $25$ million in order to make efficient use of memory, as it was advantageous in scheduling to control the total memory size of each job.
Two bits per integer were required during the sieving process, to record the parity of $n-\Omega(n)$ at location $n$, or to indicate that the value at $n$ is not yet known, so a subinterval required $6.25$ megabytes of memory.

When a process sieves an interval $[a,b]$, it first sieves using primes $b/M_S<p\leq\sqrt{b}$, working from the largest  primes downward, since removing such a prime from any integer in this interval produces a cofactor in the stored table for $[1,M_S]$, after removing any factors of $2$, $3$, and $5$.
After this we sieve with the primes $7\leq p\leq b/M_S$, starting again with the largest primes and working downwards.
The cofactor remaining after removing all factors of such a prime $p$ may still exceed $M_S$, so we employ trial division beginning with $p=2$ until the remaining cofactor is less than $M_S$, or has been determined to be prime.
After this, the remaining unknown elements in the current interval $[a,b]$ must have the form $2^i 3^j 5^k q$ with $q=1$ or $q$ prime, and some final scans take care of these.
Cumulative totals are computed at the conclusion of each subinterval, and sampled values are occasionally produced as output.
These sampled values record cumulative totals relative to the start of the large interval handled by the current process.
At its conclusion, a process records cumulative totals over its entire run, and once all processes are completed these partial sums may be used to adjust the sampled values produced across the entire computation.

The same strategy was employed to compute values of $S_\alpha(x)$ for $\alpha=1/4$, $1/2$, $3/4$, and $1$, although here we accumulated the positive and negative contributions separately to guard against loss of precision.

Our computations allowed us to establish Theorem~\ref{thmSComps}.
Record high and low values were recorded for each interval of size $8\cdot10^{10}$ handled by a process, and this produced the record values in parts (b) and (d) of this theorem.
For each of parts (a) and (c), we needed to run an additional job to determine the precise first crossing points, once the true value of each process' starting point was known.

Figures~\ref{figS0sampling} and~\ref{figS1sampling} exhibit our results for sampling over $24\leq u\leq\log(7.5\cdot 10^{14})\approx34.251$ for the functions $S_0(e^u)e^{-u/2}$ and $S_1(e^u)e^{u/2}$.
These plots also show the center of the oscillations $\pm(1+\sqrt{2})/\zeta(1/2)$, and the conjectured bounds from \eqref{Sol2} and \eqref{Sol3}.
In each plot, the crossing in the conjectured lower bounds in these inequalities near $u=34.186$ is evident.

\begin{figure}[tb]
\begin{center}
\includegraphics[width=3.5in]{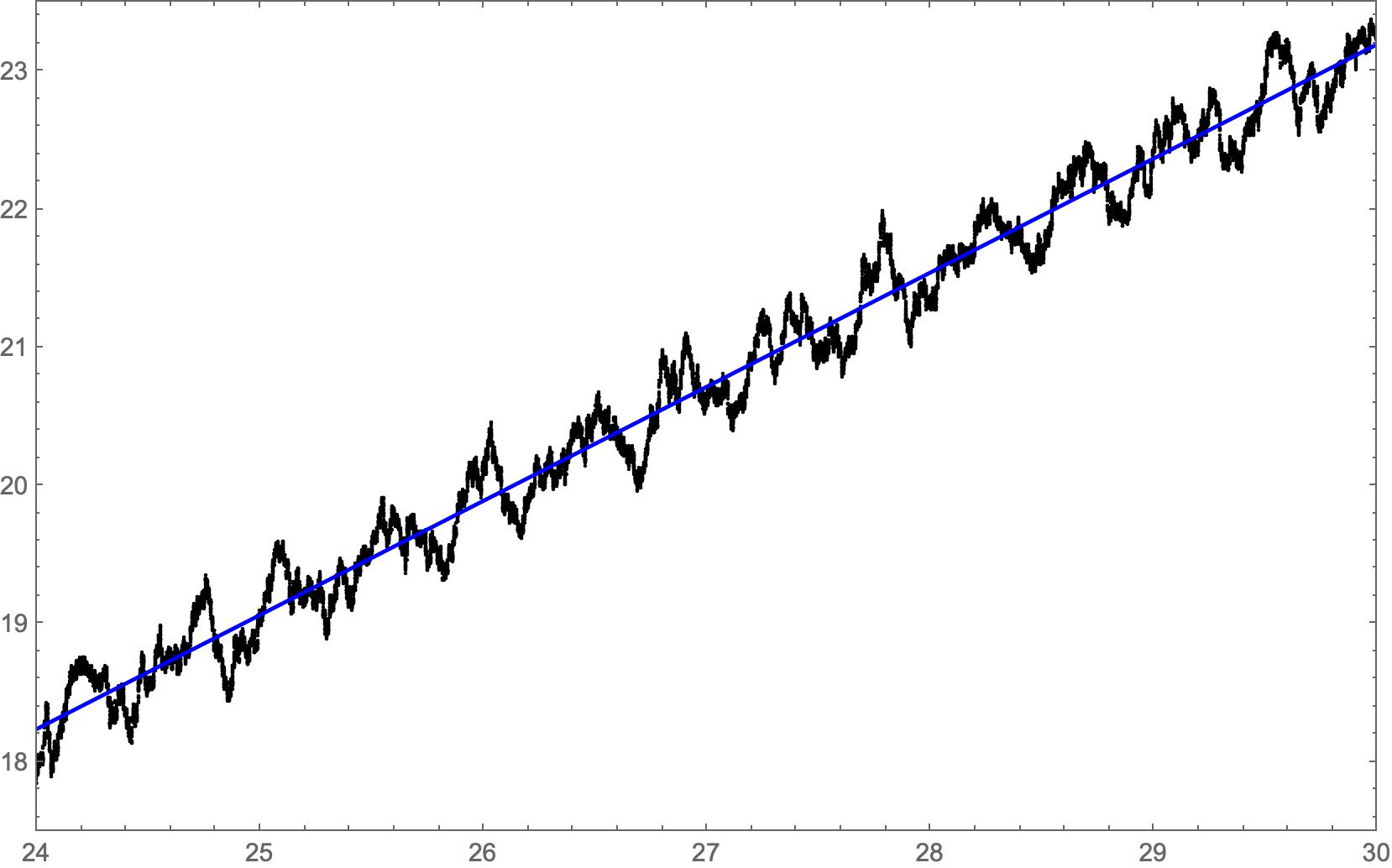}
\end{center}
\caption{$S_{1/2}(e^u)$ and the line \eqref{eqnS2center} for $24\leq u\leq30$.}\label{figSQ2plot}
\end{figure}

No violations of the upper bound in \eqref{Sol2} or \eqref{Sol3} were found in this computation; the closest point was
\[
S_0(1165833625987) = 2474979,
\]
which achieves $2.2922$, and is visible as a near-miss in Figure~\ref{figS0sampling}.
A similar near-miss occurs for $S_1(x)$ as seen in Figure~\ref{figS1sampling}.
However, further experimentation with \eqref{teak} suggested a potential location for such crossings.
Figure~\ref{figS0cross}(b) illustrates a crossing for the estimate from this expression when $T=5200$, so using $4734$ zeros of the zeta function, near $u=64.21455$, or approximately $x=7.27\cdot10^{27}$.
This is well beyond our present computational range.

Our strategy was similar for computing $H_\alpha(x)$, for a number of values of $\alpha$, with suitable adjustments to the code to track the parity of $\omega(n)$ rather than $n-\Omega(n)$.
Figure~\ref{figAH} exhibits plots for $A_\alpha(u)=\mathcal{H}_\alpha(u)e^{(\alpha-\frac{1}{2})u}$ for several values of $\alpha$.

\begin{figure}[tb]
\begin{tabular}{ccc}
\includegraphics[width=\oscwidth]{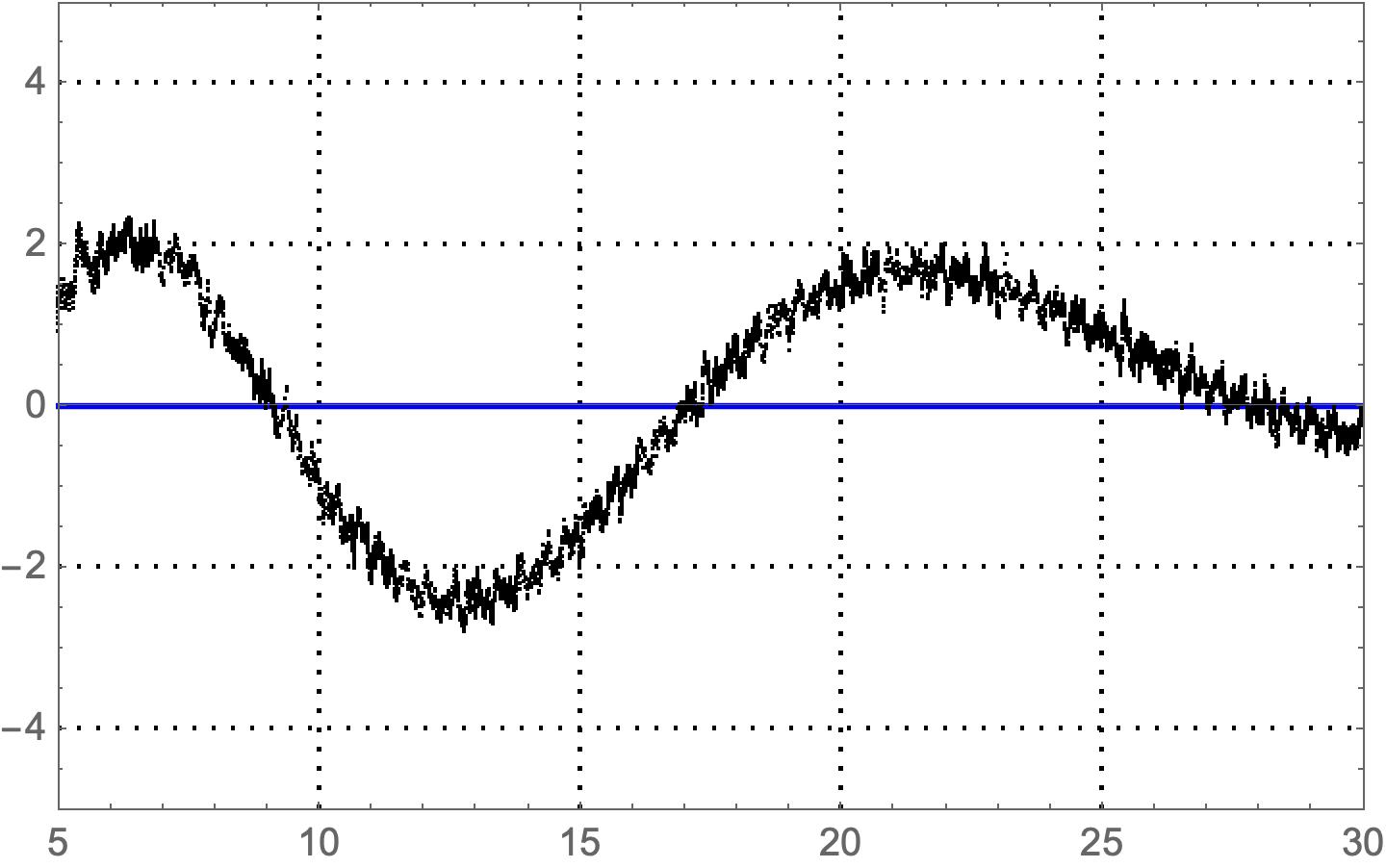} &
\includegraphics[width=\oscwidth]{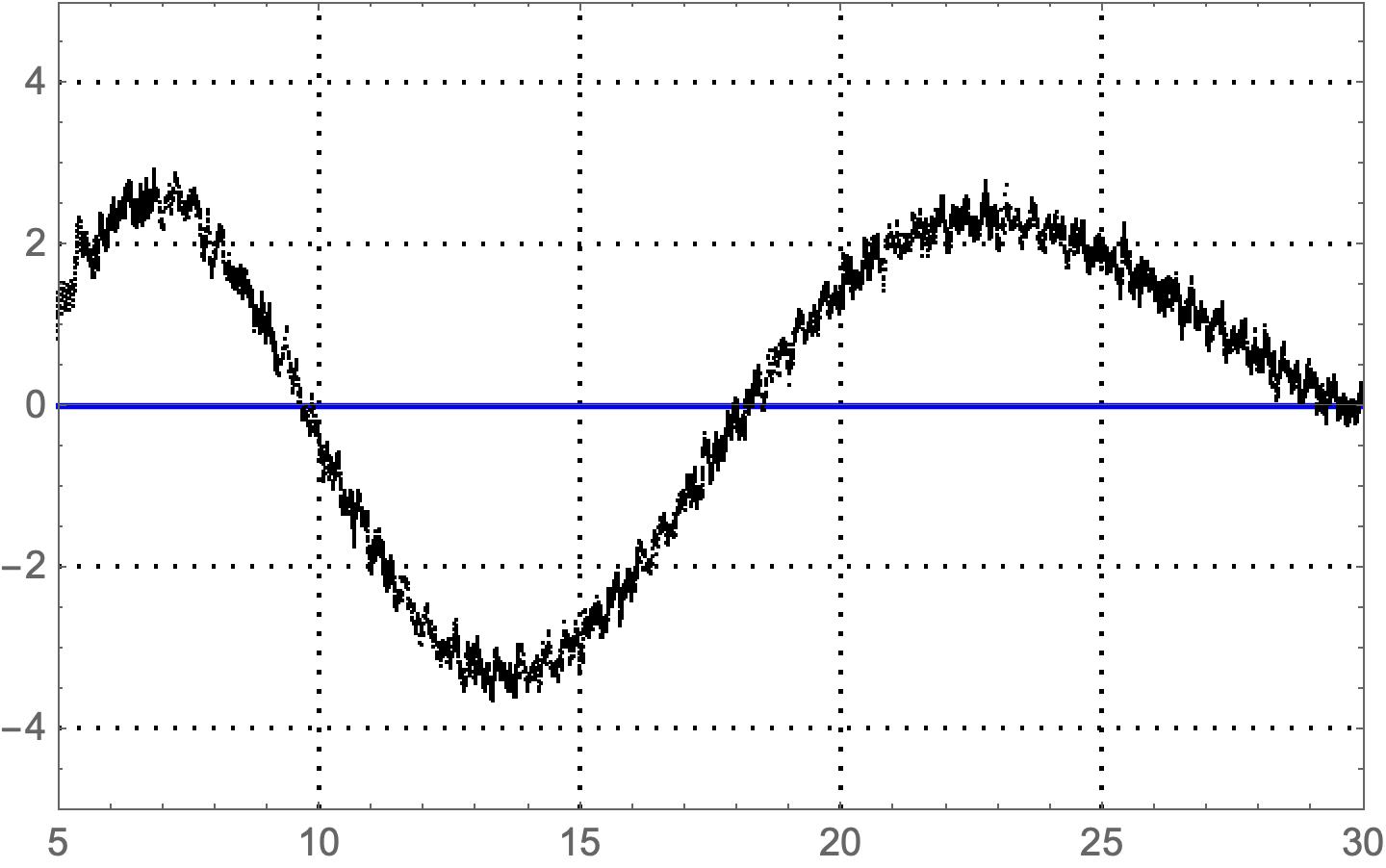} &
\includegraphics[width=\oscwidth]{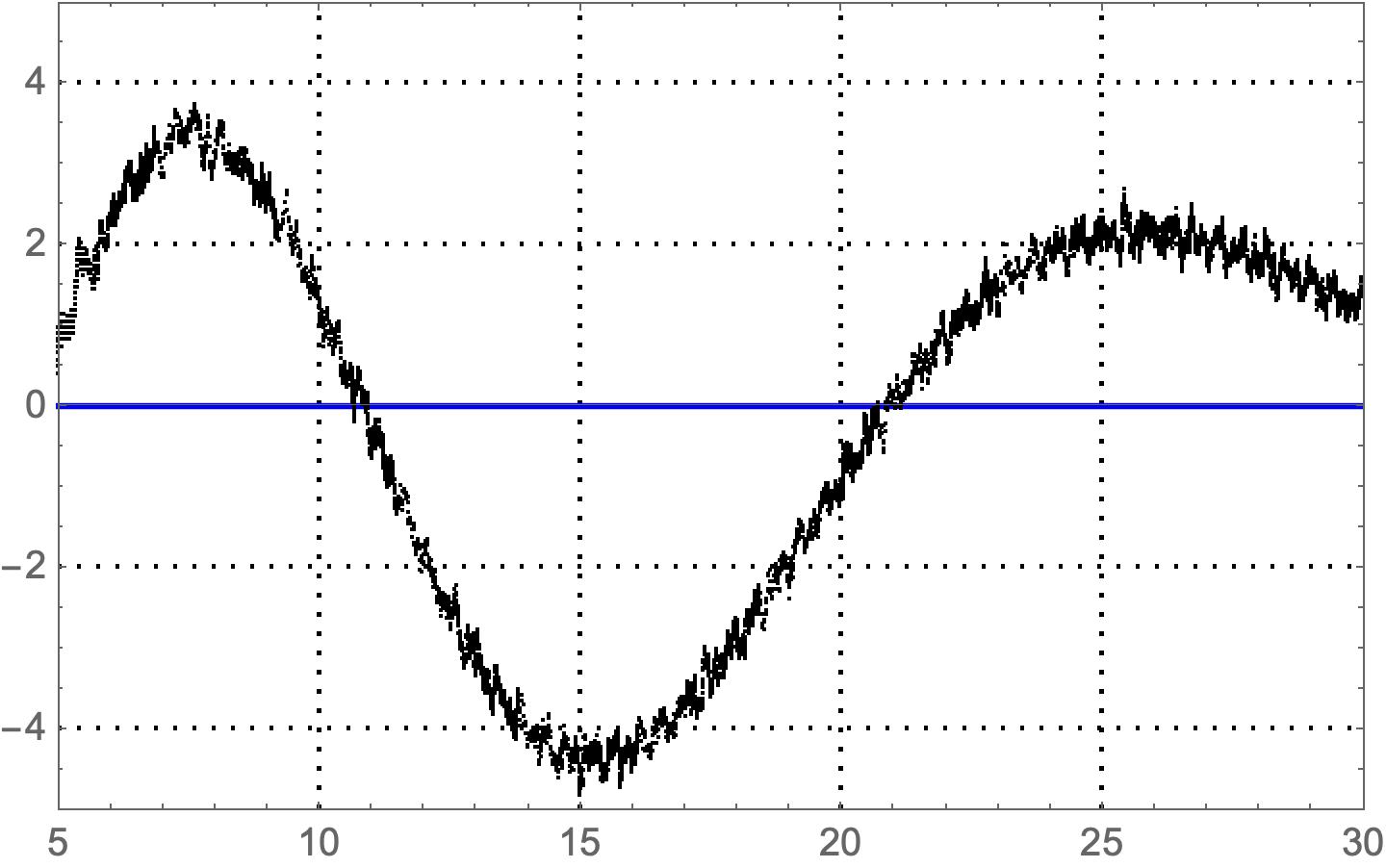} \\
(a) $\alpha=0$ &
(b) $\alpha=1/4$ &
(c) $\alpha=1/2$ \\[12pt]
\multicolumn{3}{c}{
\begin{tabular}{cc}
\includegraphics[width=\oscwidth]{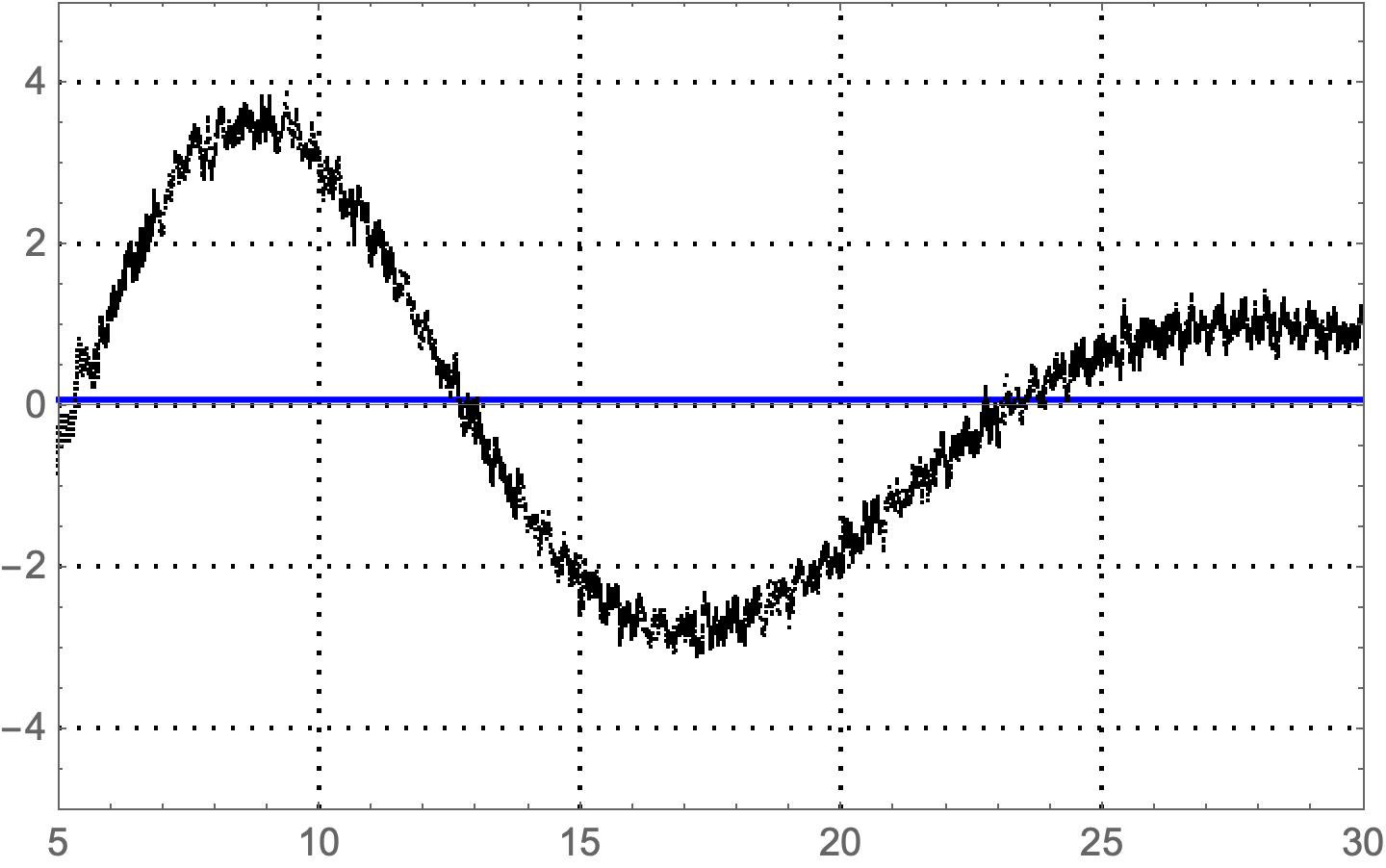} &
\includegraphics[width=\oscwidth]{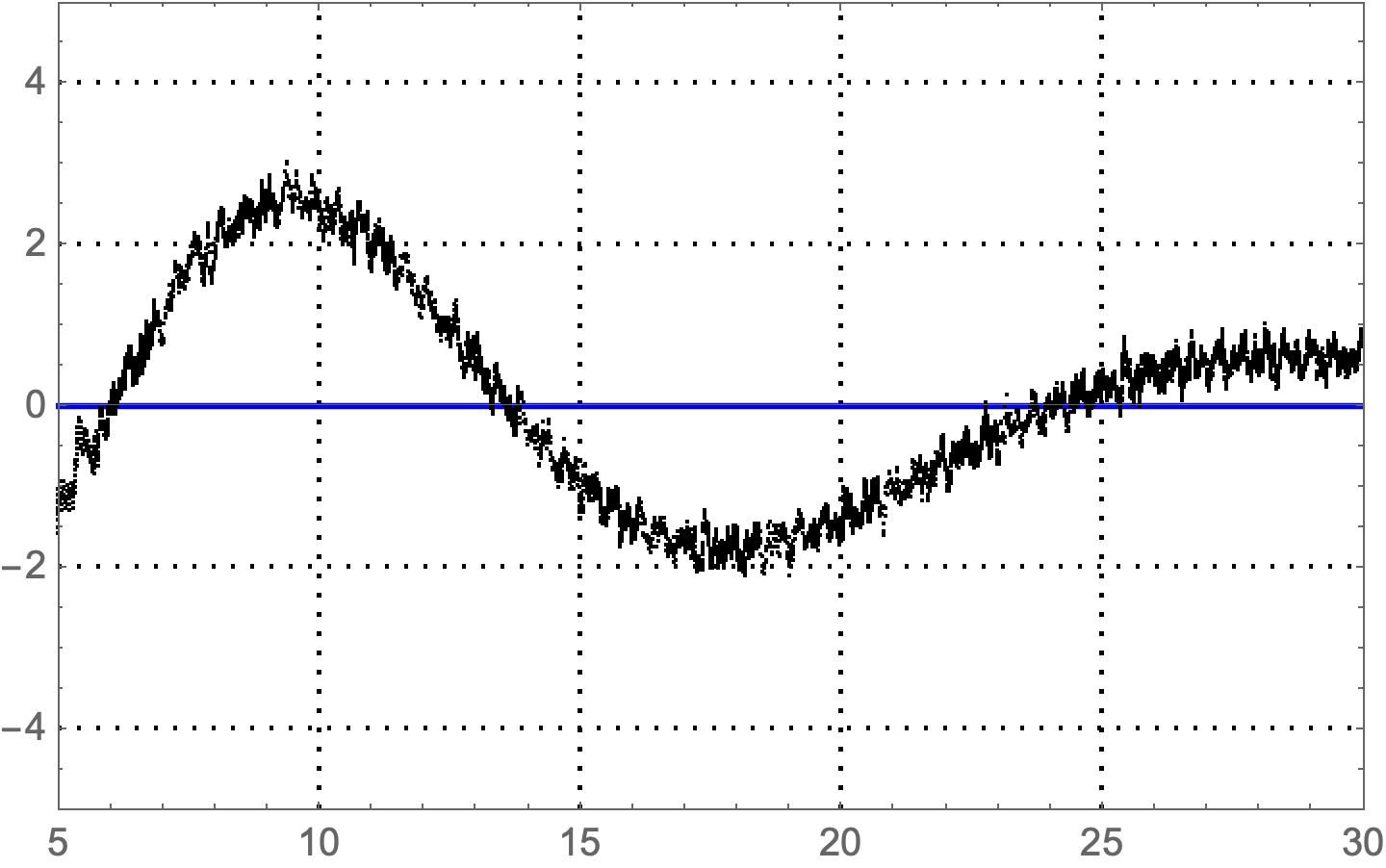} \\
(d) $\alpha=3/4$ &
(e) $\alpha=1$
\end{tabular}}\\
\end{tabular}
\caption{$A_\alpha(u)=\mathcal{H}_\alpha(u)e^{(\alpha-\frac{1}{2})u}$ for a number of values of $\alpha\in[0,1]$.}\label{figAH}
\end{figure}

\begin{figure}[tb]
\begin{center}
\begin{tabular}{cc}
\includegraphics[width=2in]{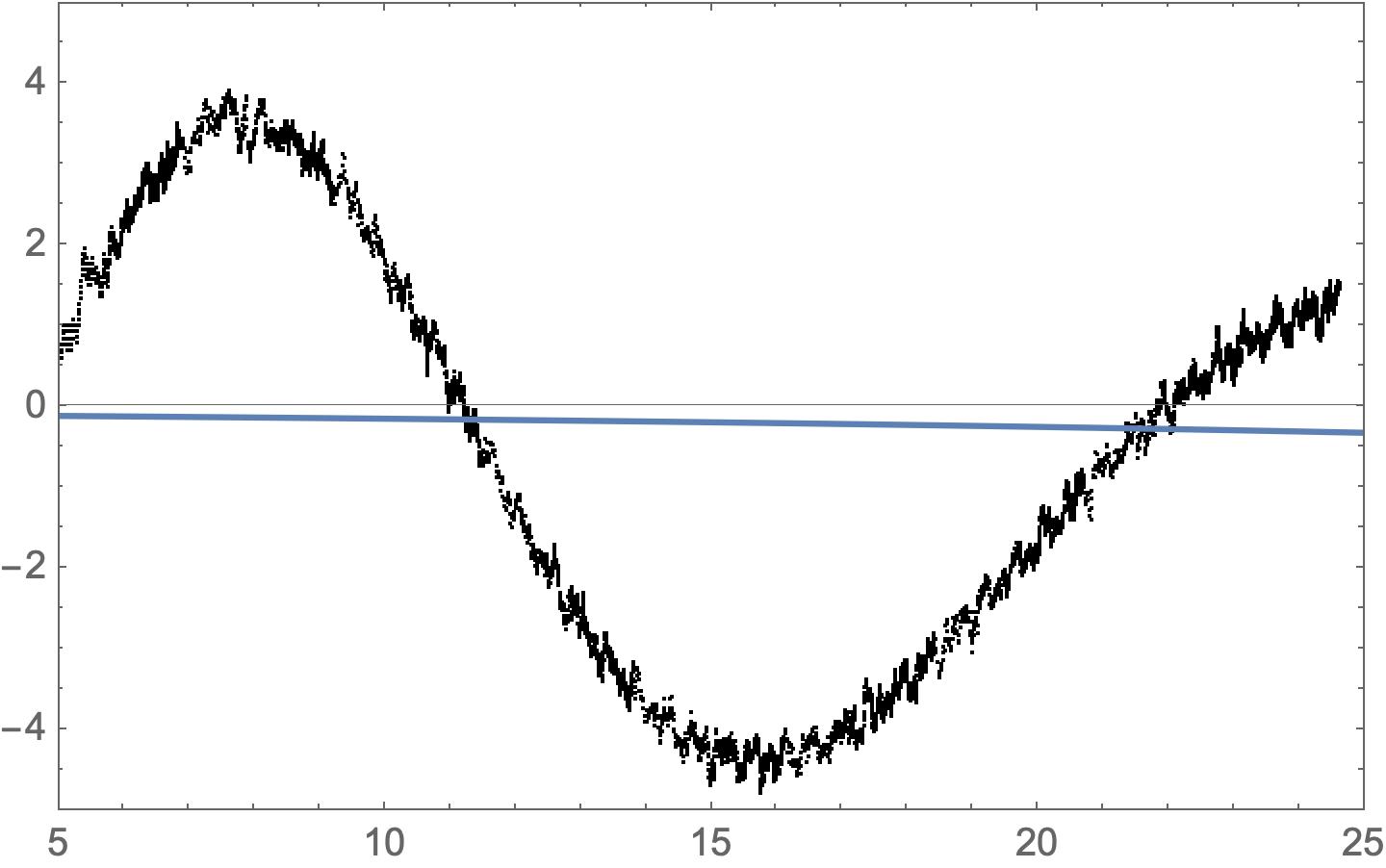} &
\includegraphics[width=2in]{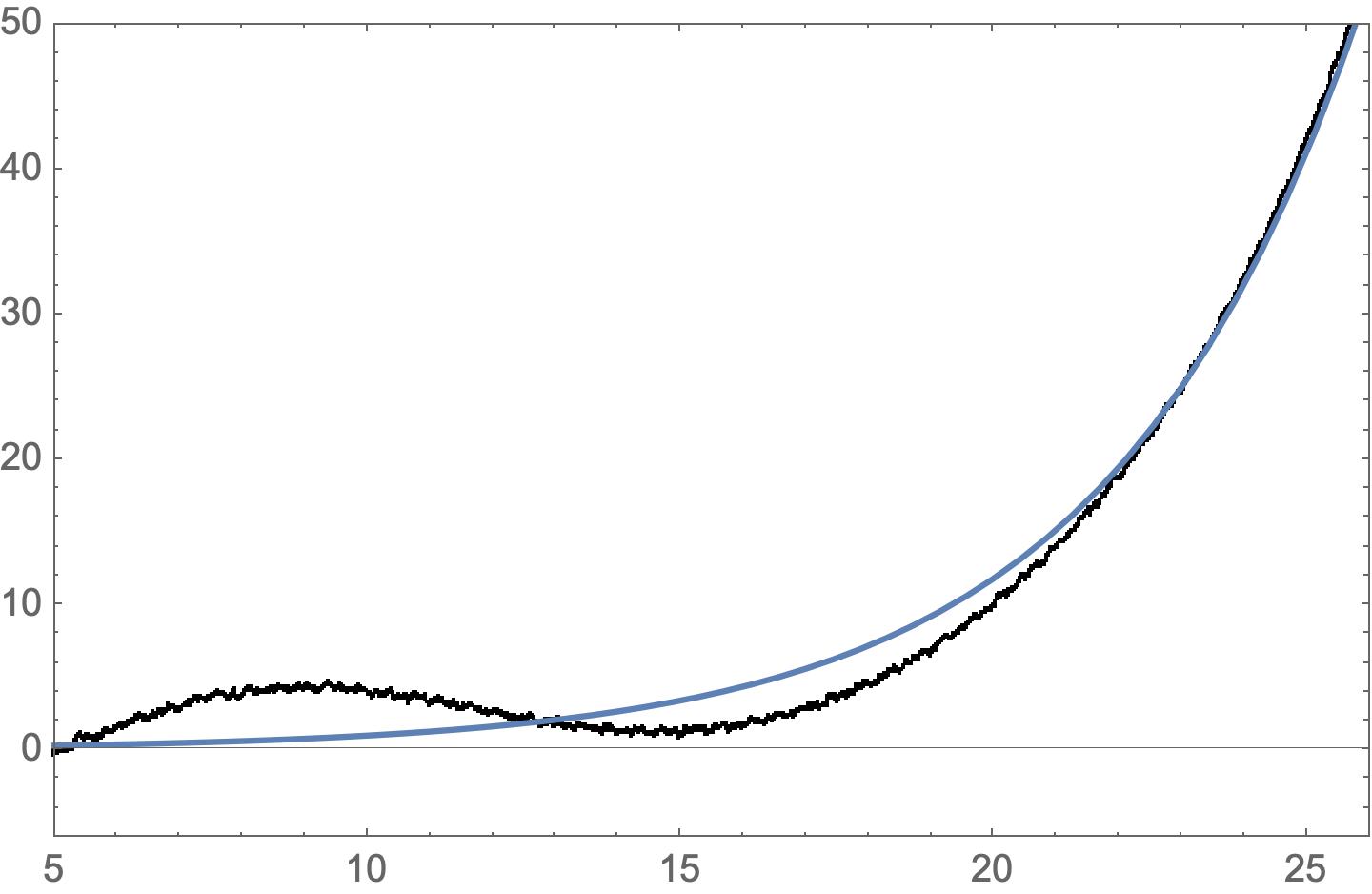}\\
(a) $\alpha=0.55$, $h(\alpha)=-0.094719\ldots$ & (b) $\alpha=0.75$, $h(\alpha)=0.079384\ldots$
\end{tabular}
\end{center}
\caption{$H_\alpha(e^u)e^{(\alpha-\frac{1}{2})u}$, without accounting for the bias $h(\alpha)$ as in \eqref{eqnHcal}, near locations of maximal bias in the negative and positive directions respectively.}\label{figHbias}
\end{figure}

A similar strategy was employed to compute values of $W(x)$ up to $2.5\cdot 10^{14}$, although here the table size was $M_W=3\cdot 10^{10}$, since we now required four bits to store the value of $\Omega(n)$ at each location in this interval, rather than simply its parity (note $\log_7(3\cdot10^{10})\approx12.4$).
This table then required about $3.72$ gigabytes of memory.
During the sieving stage, each process allocated one byte per integer in each sieving block, since $\log_2(2.5\cdot10^{14}) \approx 47.8$, so a subinterval required $25$ megabytes of storage.
We find that Sun's conjecture \eqref{Sol4} survives over this range: the maximum value of $W(x)/x$ for integers $x\in[3079, 2.5\cdot 10^{14}]$ occurs at $x=6261$, where the ratio is $0.989458\ldots$\,; the minimum value occurs at $x=3130$, where it is $-0.994568\ldots$\,.
We report a few additional values in Table~\ref{tblW}, including the best value achieved in each direction for $1.5\cdot10^{10}\leq x\leq 2.5\cdot10^{14}$.
Figure~\ref{figWsampling} displays $W(e^u)e^{-u}$ over $24\leq u\leq\log(2.5\cdot10^{14})\approx33.152$.

\begin{table}
\caption{Some local extrema for $W(x)/x$.}\label{tblW}
\begin{tabular}{|r|r|r|}\hline
\multicolumn{1}{|c|}{$x$} & \multicolumn{1}{|c|}{$W(x)$} & \multicolumn{1}{|c|}{$W(x)/x$}\\\hline
$3130$ & $-3113$ & $-0.994568\ldots$\\
$6261$ & $6195$ & $0.989458\ldots$\\
$6410313$ & $6316905$ & $0.985428\ldots$\\
$12820626$ & $-12574965$ & $-0.980838\ldots$\\
$52513285735$ & $-51364764131$ & $-0.978128\ldots$\\
$105026571390$ & $102737766207$ & $0.978207\ldots$\\\hline
\end{tabular}
\end{table}

\begin{figure}[tb]
\begin{center}
\includegraphics[width=3.5in]{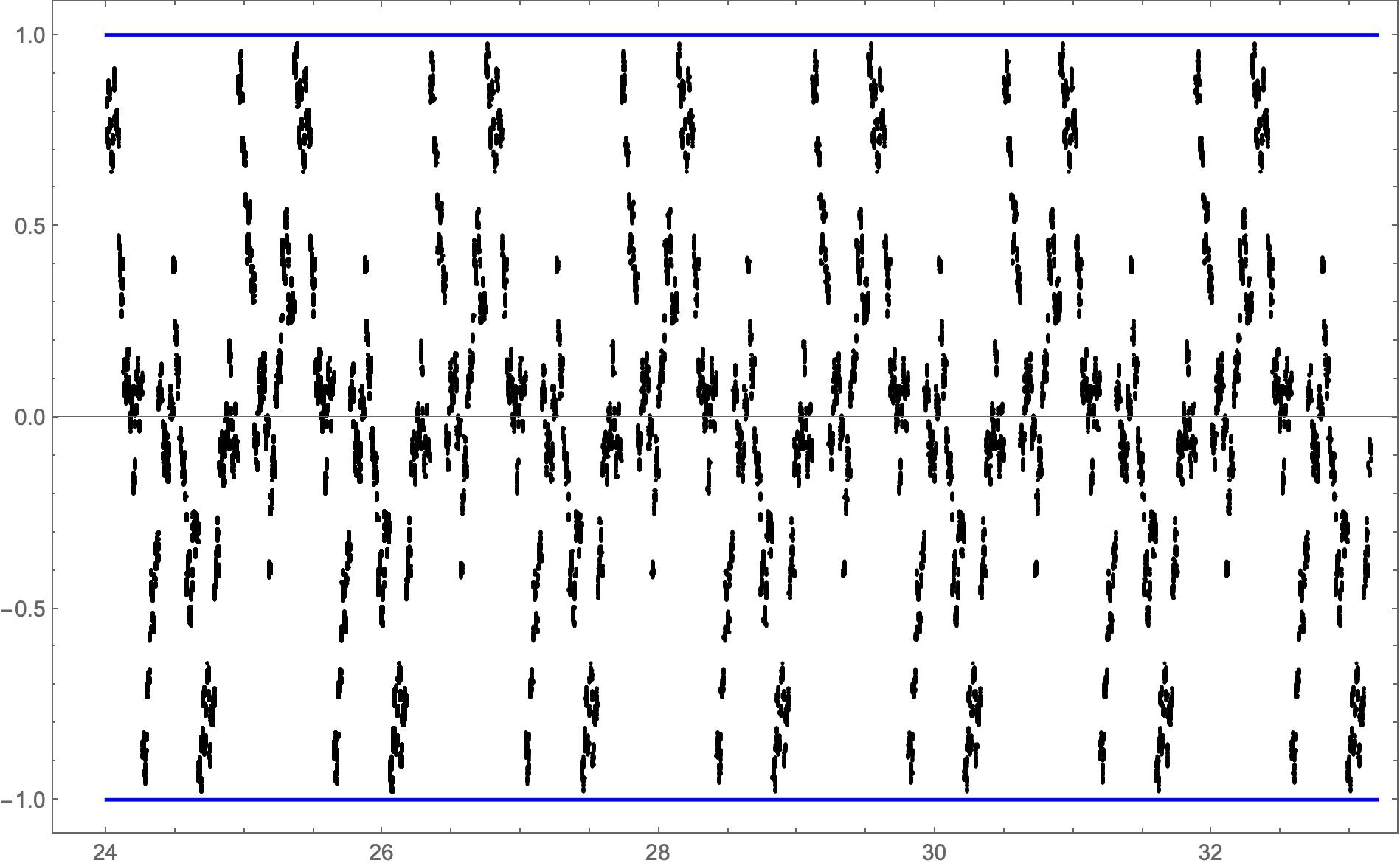}
\caption{Sampled values for $W(e^u)e^{-u}$.}\label{figWsampling}
\end{center}
\end{figure}

We note that the function $W(x)$ may be amenable to some further analysis.
As in \eqref{eqnhDef}, the corresponding Dirichlet function in this case is
\begin{equation*}\label{ebony}
J(s) = \sum_{n=1}^{\infty} \frac{(-2)^{\Omega(n)}}{n^{s}} = \prod_{p} \left( 1 + \frac{2}{p^{s}}\right)^{-1}
\end{equation*}
for $\sigma >1$, and using methods similar to that employed to create \eqref{chalk} for $h(s)$, one may show that
\begin{equation*}\label{lid}
J(s) = \frac{\zeta(2s)^{3} \zeta(4s)^{3} J_4(s)}{\zeta(s)^{2} \zeta(3s)^{2}},
\end{equation*}
where
\[
J_4(s) = \prod_p \frac{(1-p^{-2s})^3(1 - p^{-4s})^3}{(1 + 2p^{-s})(1-p^{-s})^2(1- p^{-3s})^2},
\]
which converges for $\sigma>1/5$.
The function $J(s+1/2)/(s+1/2)$ has poles of order~$2$ on the imaginary axis, as well as one of order $3$ at the origin. We can manipulate this to ensure that the result of Anderson and Stark \eqref{eqnAndStark} is applicable. However, some additional complications arise, which we hope to address in future work.

\section{Other conjectures}\label{secMoreConj}

Sun made two additional families of conjectures concerning the quantity $n-\Omega(n)$.
First, in \cite[Conj.\ 1.2]{Sun}, Sun opined that for $m=3$, $5\leq m\leq18$, or $m=20$, the proportion of integers $n\leq x$ having the property that $m\mid(n-\Omega(n))$ exceeds $1/m$, provided that $x\geq s(m)$, for a particular quantity $s(m)$.
For example, it was conjectured that $s(3)=62$, $s(5)=187$, and $s(20)=61$ all suffice.
For $m=4$, this proportion was posited to be bounded above by $1/4$, for $x\geq1793193$.
No conjecture was made in the case $m=19$, although it was remarked that perhaps this proportion is infinitely often less than $1/19$, and infinitely often greater.
We verified that Sun's seventeen conjectures here hold for $x\leq10^{11}$.
We can also report that the last crossing in the case $m=19$ over this range occurs at $x=49675549593$:
from here to $x=10^{11}$ the proportion of integers $n\leq x$ having $19\mid(n-\Omega(n))$ exceeds $1/19$.

Second, in \cite[Conj.\ 1.3]{Sun} Sun stated a conjecture regarding a variant of $S_0(x)$ twisted by Dirichlet characters.
For $d\equiv 0$ or $1$ mod~$4$, let $(\frac{d}{n})$ denote the Kronecker symbol, and let
\[
S_d(x) = \sum_{n\leq x} (-1)^{n - \Omega(n)} \left(\frac{d}{n}\right).
\]
Sun posited for example that
\begin{equation*}\label{Set}
S_{-4}(x)<0, \quad S_{-7}(x) < 0, \quad S_{-3}(x)>0, \quad  S_5(x) > 0,
\end{equation*}
with $x\geq 11$ sufficing in each case.
In fact, additional conjectures of this form for other values of $d$ were stated as well. These could all be investigated by using results on $N$-independence of zeros of the Dirichlet $L$-function $L(s, \chi)$ where $\chi$ is the non-principal character modulo $4$.
We know of no calculation, similar to that for $\zeta(s)$ in \cite{Bateman,Best}, that establishes good $N$-independence for Dirichlet $L$-functions, although Grosswald \cite[Thm.\ 5 and \S8, \S9]{Grosswald72} presents some calculations to this end.
Such a program of research appears both possible and interesting: we hope to return to this in future work.

\section*{Acknowledgments}

We thank Peter Humphries for bringing these questions to our attention, and we thank the referee for a number of helpful comments.
We also thank NCI Australia and UNSW Canberra for computational resources.
This research was undertaken with the assistance of resources and services from the National Computational Infrastructure (NCI), which is supported by the Australian Government.

\bibliographystyle{amsplain}

\end{document}